\documentclass[aop,MSNbibl,seceqn,dvips]{arximspdf}

\doi{10.1214/11-AOP671}
\volume{40}
\issue{3}
\pubyear{2012}
\firstpage{1357}
\lastpage{1374}

\makeatletter
\newtheorem{theorem}{Theorem}
\newproclaim{defn}[theorem]{Definition}
  \newtheorem{lem}[theorem]{Lemma}
  \newproclaim{rem}[theorem]{Remark}
  \newtheorem{prop}[theorem]{Proposition}

\def\bsuffix #1{#1}

\makeatother

\begin{document}
\begin{frontmatter}

\title{Maharam extension and stationary stable processes}
\runtitle{Maharam extension and stationary stable processes}

\begin{aug}
\author{\fnms{Emmanuel} \snm{Roy}\corref{}\ead[label=e1]{roy@math.univ-paris13.fr}}
\runauthor{E. Roy}
\affiliation{Universit\'{e} Paris}
\address{Laboratoire Analyse G\'{e}om\'{e}trie et Applications\\
UMR 7539\\
Universit\'{e} Paris 13\\
99 avenue J.-B. Cl\'{e}ment\\
F-93430 Villetaneuse\\
France\\
\printead{e1}} 
\end{aug}

\received{\smonth{1} \syear{2010}}
\revised{\smonth{3} \syear{2011}}

%
\begin{abstract}
We give a second look at stationary stable processes by interpreting
the self-similar property at the level of the L\'{e}vy measure as characteristic
of a Maharam system. This allows us to derive structural results and
their ergodic consequences.
\end{abstract}

%

\end{frontmatter}

\section{Introduction}

In a fundamental paper \cite{Ros95strucsta}, Rosi{\'n}ski
revealed the hidden structure of stationary symmetric $\alpha$-stable
($S\alpha S$) processes. Namely, he proved that, following Hardin \cite{Hard82mini}, through what is
called a \textit{minimal spectral
representation}, such a process is driven by a nonsingular dynamical
system.

Such a result was proved to classify those processes according to
their ergodic properties such as various kinds of mixing. In \cite{Roy06IDstat},
we used a different approach as we considered the whole family of
stationary infinitely divisible processes without Gaussian part (called
\textit{IDp processes}). The key tool there was the L\'{e}vy measure system
of the process, which was measure-preserving and not just merely nonsingular.
So far, in the stable case, the connection between the L\'{e}vy measure and the
nonsingular system was not clear. This is
the purpose of this paper, to fill the gap and go beyond both approaches.

Indeed, we will prove that L\'{e}vy measure systems of $\alpha$-stable
processes have the form of a so-called Maharam system. This observation
has some interesting consequences as it allows us to derive very quickly
minimal spectral representations in the $S\alpha S$ case, to reinforce
factorization results, and to refine ergodic classification.

Let us explain very loosely the mathematical features of stable distributions
we will be using. Observe that stable distributions are characterized
by a self-similar property which is obvious when observing the corresponding
L\'{e}vy process:

If $X_{t}$ is an $\alpha$-stable L\'{e}vy process, then $b^{-{1}/{\alpha}}X_{bt}$
has the same distribution.

However, if not obvious or useful, this property is also present for
any $\alpha$-stable object but takes another form. The common feature
is to be found in the L\'{e}vy measure:

Loosely speaking, if $\{ X_{t}\} _{t\in S}$ is an $\alpha$-stable
process indexed by a set $S$, then for any positive number $c$,
the image of the L\'{e}vy measure $Q$ by the map $R_{c}:=\{ x_{t}\}
_{t\in S}\mapsto\{ cx_{t}\} _{t\in S}$
is $c^{-\alpha}Q$.

This property of the L\'{e}vy measure is characteristic of $\alpha$-stable
processes and can be translated into an ergodic theoretic statement:

The measurable nonsingular flow $\{ R_{c}\} _{c\in\mathbb{R}_{+}}$
is dissipative and the multiplicative coefficient $c^{-\alpha}$ has
an outstanding importance in this matter, since it reveals the structure
of a Maharam transformation. The importance is even greater when there
is more invariance involved (stationary $\alpha$-stable processes,
etc$.\ldots$), as in the present paper.

The paper is organized as follows. In Section \ref{secSpectral-representation}
we recall what a spectral representation is, and in Section \ref{secSome-terminology},
we give the necessary background in nonsingular ergodic theory. Maharam
systems are introduced in Section \ref{secMaharam-transformation},
and the link with L\'{e}vy measures of stable processes, together with
spectral representations is exposed in Section \ref{secL=0000E9vy-measure-as}.
Section \ref{secRefinements-of-the} is a refinement of the structure
of stable processes. We deduce from the preceding results some ergodic
properties in Section \ref{secErgodic-properties}.

\section{Spectral representation}\label{secSpectral-representation}

We warn the reader that we will, most of the time, omit the implicit
``$\mu$-a.e.'' or ``modulo null sets'' throughout the document.

A family of functions $\{ f_{t}\} _{t\in T}\subset L^{\alpha}(\Omega
,\mathcal{F},\mu)$
where $(\Omega,\mathcal{F},\mu)$ is a $\sigma$-finite
Lebesgue space is said to be a spectral representation of $S\alpha S$
process $\{ X_{t}\} _{t\in T}$ if
\[
\{ X_{t}\} _{t\in T}=\biggl\{ \int_{\Omega}f_{t}(\omega)M(\mathrm{d}\omega)\biggr\}
_{t\in T}
\]
holds in distribution, $M$ being an independently scattered $S\alpha S$-random
measure on $(\Omega,\mathcal{F})$ with intensity measure
$\mu$.

We will say that a spectral representation is \textit{proper} if
$\operatorname{Supp}\{ f_{t}, t\in T\} =\Omega$.
Of course we obtain a proper representation from a general one by
removing the complement of $\operatorname{Supp}\{ f_{t}, t\in T\} $.

To express that a representation contains the strict minimum to define
the process, the notion of minimality has been introduced (Hardin
\cite{Hard82mini}):

A spectral representation is said to be $\{ f_{t}\} _{t\in T}\subset
L^{\alpha}(\Omega,\mathcal{F},\mu)$
\textit{minimal} if it is proper and $\sigma(\frac{f_{t}}{f_{s}}1_{\{
f_{s}\neq0\} }, s,t\in T)=\mathcal{F}$.

Hardin proved in \cite{Hard82mini} the existence of minimal representations
for $S\alpha S$ processes.

In the stationary case ($T=\mathbb{R} \mbox{ or } \mathbb{Z}$),
Rosi{\'n}ski
has explained the form of the spectral representation:
\begin{theorem}[(Rosi{\'n}ski)] Let $\{ f_{t}\} _{t\in T}\subset L^{\alpha
}(\Omega,\mathcal{F},\mu)$
be a minimal representation of a stationary $S\alpha S$-process;
then there exists nonsingular flow $\{ \phi_{t}\} _{t\in T}$
on $(\Omega,\mathcal{F},\mu)$ and a cocycle $\{ a_{t}\} _{t\in T}$
for this flow with values in $\{ -1,1\} $ (or in $|z|=1$
in the complex case) such that, for each $t\in T$,
\[
f_{t}=a_{t}\biggl\{ \frac{\mathrm{d}\mu\circ\phi_{t}}{\mathrm{d}\mu}\biggr\} ^{{1}/{\alpha}}(f_{0}\circ\phi_{t}).
\]

\end{theorem}

\section{Some terminology}\label{secSome-terminology}

A quadruplet $(\Omega,\mathcal{F},\mu,T)$ is called a
\textit{dynamical system} or shortly a \textit{system} if $T$ is a
\textit{nonsingular
automorphism}, that is, a bijective bi-measurable map such that $T^{*}\mu
\sim\mu$.
If $T_{*}(\mu)=\mu$, then $(\Omega,\mathcal{F},\mu,T)$
is a \textit{measure-preserving} (abr. \textit{m.p.}) dynamical system.

A system $(\Omega_{2},\mathcal{F}_{2},\mu_{2},T_{2})$
is said to be a \textit{nonsingular (resp. measure preserving) factor}
of the system $(\Omega_{1},\mathcal{F}_{1},\mu_{1},T_{1})$
if there exists a measurable \textit{nonsingular (resp. measure preserving)
homomorphism} between them, that is, a measurable map $\Phi$ from
$\Omega_{1}$ to $\Omega_{2}$ such that $\Phi T_{1}=T_{2}\Phi$ and
$\Phi^{*}\mu_{1}\sim\mu_{2}$ (resp. $\Phi^{*}\mu_{1}=\mu_{2}$).
If $\Phi$ is invertible and bi-measurable it is called a \textit{nonsingular
(resp. measure preserving) isomorphism}, and the system is said to
be \textit{nonsingular (resp. measure preserving) isomorphic}.

\subsection{Krieger types}

Consider a nonsingular dynamical system $(\Omega,\mathcal{F},\mu,T)$.
A set $A\in\mathcal{F}$ such that $\mu(A)>0$ is said
to be \textit{periodic} of period $n$ if $T^{i}A$, $0\leq i\leq n-1$,
are disjoint and $T^{n}A=A$ and \textit{wandering} if $T^{i}A$, $i\in
\mathbb{Z}$
are disjoint. A set is \textit{exhaustive} if $\bigcup_{k\in\mathbb
{Z}}T^{k}A=\Omega$.
A system is \textit{conservative} if there is no wandering set and
\textit{dissipative}
if there is an exhaustive wandering set.

$(\Omega,\mathcal{F},\mu,T)$ is said to be of \textit{Krieger
type}:
\begin{itemize}
\item$I_{n}$ if there exists an exhaustive set of period $n$;
\item$I_{\infty}$ if it is dissipative;
\item$\mathrm{II}_{1}$ if there is no periodic set and exists an equivalent
finite $T$-invariant measure;
\item$\mathrm{II}_{\infty}$ if is is conservative with an equivalent infinite
$T$-invariant continuous measure but no absolutely continuous finite
$T$-invariant measure;
\item$\mathrm{III}$ if there is no absolutely continuous $T$-invariant
measure.
\end{itemize}

\section{Maharam transformation}\label{secMaharam-transformation}
\begin{defn}
An m.p. dynamical system is said to be \textit{Maharam} if it is isomorphic
to $(\Omega\times\mathbb{R},\mathcal{F}\otimes\mathcal{B},\mu\otimes
\mathrm{e}^{s}\,\mathrm{d}s,\widetilde{T})$,
where $T$ is a nonsingular automorphism of $(\Omega,\mathcal{F},\mu)$,
and $\widetilde{T}$ is defined by
\[
\widetilde{T}(\omega,s):=\biggl(T(\omega),s-\ln\frac{\mathrm{d}T_{*}^{-1}\mu
}{\mathrm{d}\mu}(\omega)\biggr).
\]
\end{defn}
Observe that the dissipative flow $\{ \tau_{t}\} _{t\in\mathbb{R}}$
defined by $\tau_{t}:=(\omega,s)\mapsto(\omega,s-t)$
commutes with $\widetilde{T}$.

Note that we have chosen the usual additive representation but we
could (and eventually will!) use the following multiplicative representation
of a~Maharam system. Take $0<\alpha<2$. We can represent $(\Omega\times
\mathbb{R},\mathcal{F}\otimes\mathcal{B},\mu\otimes\mathrm{e}^{s}\,\mathrm
{d}s,\widetilde{T})$
by the system $(\Omega\times\mathbb{R}_{+}^{*},\mathcal{F}\otimes
\mathcal{B}_{+},\mu\otimes\frac{1}{s^{1+\alpha}}\,\mathrm{d}s,\widetilde
{T}_{\alpha})$,
where $\widetilde{T}_{\alpha}$ is defined by
\[
\widetilde{T}_{\alpha}(\omega,s):=\biggl(T(\omega),s\biggl(\frac{\mathrm
{d}T_{*}^{-1}\mu}{\mathrm{d}\mu}(\omega)\biggr)^{{1}/{\alpha}}\biggr).
\]

The isomorphism is provided by the map $(\omega,s)\mapsto(\omega
,(2-\alpha)^{-{1}/{(2-\alpha)}}{e}^{(2-\alpha)s})$.
Observe that, under this isomorphism, $\{ \tau_{t}\} _{t\in\mathbb{R}}$
is changed into $\{ S_{\mathrm{e}^{{t}/{\alpha}}}\} _{t\in\mathbb
{R}_{+}^{*}}$,
where~$S_{t}$ is the multiplication by $t$ on the second coordinate.

In \cite{AarLemVol98salad}, the authors proved the following characterization,
as a straightforward application of Krengel's representation of dissipative
transformations:

\begin{theorem}
\label{thmMaharam}A system $(X,\mathcal{A},\nu,\gamma)$
is Maharam if and only if there exists a measurable flow $\{ Z_{t}\}
_{t\in\mathbb{R}}$
commuting with $\gamma$ such that $(Z_{t})_{*}\nu=\mathrm{e}^{t}\nu$.
$\{ Z_{t}\} _{t\in\mathbb{R}}$ corresponds to $\{ \tau_{t}\} _{t\in
\mathbb{R}}$
under the isomorphism with the Maharam system under the additive representation.
\end{theorem}

In the original theorem they assumed ergodicity of $\gamma$ to prove
that the resulting nonsingular transformation $T$ in the above representation
was actually living on a nonatomic measure space $(\Omega,\mathcal
{F},\mu)$.
The ergodicity assumption is therefore not necessary in the way we
present this theorem.

We end this section with a very natural lemma which is part of folklore.
We omit the proof.
\begin{lem}
\label{lemTwoMaharam}Consider two Maharam systems $(\Omega_{1}\times
\mathbb{R}_{+}^{*},\mathcal{F}_{1}\otimes\mathcal{B},\mu_{1}\otimes\frac
{1}{s^{1+\alpha}}\,\mathrm{d}s,\allowbreak \widetilde{T_{1}})$
and $(\Omega_{2}\times\mathbb{R}_{+}^{*},\mathcal{F}_{2}\otimes\mathcal
{B},\mu_{2}\otimes\frac{1}{s^{1+\alpha}}\,\mathrm{d}s,\widetilde{T_{2}})$,
and denote by $\{ S_{t}\} _{t\in\mathbb{R}_{+}^{*}}$ and
$\{ Z_{t}\} _{t\in\mathbb{R}_{+}^{*}}$ their respective
multiplicative flows. Assume there exists a (measure-preserving) factor
map (resp. isomorphism) $\Phi$ between the two systems such that,
for all $t\in\mathbb{R}_{+}^{*}$, $S_{t}\Phi=Z_{t}\Phi$. Then $\Phi$
induces a nonsingular factor map (resp. isomorphism) $\phi$ between
$(\Omega_{1},\mathcal{F}_{1},\mu_{1},T_{1})$ and $(\Omega_{2},\mathcal
{F}_{2},\mu_{2},T_{1})$.
\end{lem}
\begin{rem}
Observe also that the Maharam systems associated to $(\Omega,\mathcal
{F},\allowbreak \mu_{1},  T)$
and $(\Omega,\mathcal{F},\mu_{2},T)$ where $\mu_{1}\sim\mu_{2}$
are isomorphic.
\end{rem}

\subsection{\texorpdfstring{Refinements of type $\mathrm{III}$ (see \cite{DanSilv09NonSi})}{Refinements of type $\mathrm{III}$
(see [3])}}\label{subType3}

Since the flow $\{ S_{t}\} _{t\in\mathbb{R}}$ commutes
with~$\widetilde{T}$, it acts nonsingularly on the space $(Z,\nu)$
of ergodic components of~$\widetilde{T}$ and is called the \textit{associated
flow} of $T$. This flow is ergodic whenever $T$ is ergodic, and its
form allows us to classify ergodic type $\mathrm{III}$ systems:
\begin{itemize}
\item$T$ is of type $\mathrm{III}_{\lambda}$, $0<\lambda<1$, if the associated
flow is the periodic flow $x\mapsto x+t
\operatorname{mod}(-\log\lambda)$;\vadjust{\goodbreak}
\item$T$ is of type $\mathrm{III}_{0}$, if the associated flow is free;
\item$T$ is of type $\mathrm{III}_{1}$, if the associated flow is the
trivial flow on a singleton.
\end{itemize}
In particular $\widetilde{T}$ is ergodic if and only if $T$ is of
type $\mathrm{III}_{1}$.\

\section{\texorpdfstring{L\'{e}vy measure as Maharam system
and spectral representations}{Levy measure as Maharam system
and spectral representations}}\label{secL=0000E9vy-measure-as}

\subsection{\texorpdfstring{L\'{e}vy measure of $\alpha$-stable processes}{Levy measure of alpha-stable processes}}

For simplicity we will only consider discrete time stationary processes.

Let us recall, following \cite{Maruyama70IDproc} (see also \cite{Roy06IDstat}),
that the L\'{e}vy measure of stationary IDp process $X$ of distribution
$\mathbb{P}$ is the shift-invariant $\sigma$-finite measure on $\mathbb
{R}^{\mathbb{Z}}$,
$Q$, such that $Q(0_{\mathbb{R}^{\mathbb{Z}}})=0$, $\int_{\mathbb
{R}^{\mathbb{Z}}}(x_{0}^{2}\wedge1)Q(\mathrm{d}\{ x_{n}\} _{n\in\mathbb
{Z}})<\infty$
and
\begin{eqnarray*}
& & \mathbb{E}\Biggl[\exp\Biggl(i\sum_{k=n_{1}}^{n_{2}}a_{k}X_{k}\Biggr)\Biggr]\\
& &\qquad{} =\exp\Biggl[\int_{\mathbb{R}^{\mathbb{Z}}}\Biggl(\exp\Biggl(i\sum
_{k=n_{1}}^{n_{2}}a_{k}x_{k}\Biggr)-1-i\sum
_{k=n_{1}}^{n_{2}}a_{k}c(x_{k})\Biggr)Q(\mathrm{d}\{ x_{n}\} _{n\in\mathbb{Z}})\Biggr]
\end{eqnarray*}
for any choice of $-\infty<n_{1}\le n_{2}<+\infty$, $\{ a_{k}\}
_{n_{1}\le n_{2}}\in\mathbb{R}^{n_{2}-n_{1}}$.

$c$ is defined by:
\begin{eqnarray*}
c(x)&=&-1 \qquad\mbox{if } x<-1;\\
c(x)&=&x \qquad\mbox{if } -1\leq x\leq1;\\
c(x)&=&1 \qquad\mbox{if } x>1.
\end{eqnarray*}

The system $(\mathbb{R}^{\mathbb{Z}},\mathcal{B}^{\otimes\mathbb{Z}},Q,S)$
where $S$ is the shift on $\mathbb{R}^{\mathbb{Z}}$ is called the
\textit{L\'{e}vy measure system} associated to the process $X$.

The $\alpha$-stable stationary processes, $0<\alpha<2$, are (see
Chapter 3 in \cite{Sato99LevPro}) completely characterized as those
IDp processes such that their L\'{e}vy measure satisfies
%
\begin{equation}
(R_{t})_{*}Q=t^{-\alpha}Q\label{eqStable}
\end{equation}
for any positive $t$, $R_{t}$ being the multiplication by $t$,
that is,
\[
\{ x_{n}\} _{n\in\mathbb{Z}}\mapsto\{ tx_{n}\} _{n\in\mathbb{Z}}.
\]

We also recall the fundamental result of Maruyama that allows to represent
any IDp process with L\'{e}vy measure $Q$ as a stochastic integral with
respect to a Poisson measure with intensity $Q$.
\begin{theorem}[(Maruyama representation \cite{Maruyama70IDproc})]\label{thmMaruyama-rep}
Let $\mathbb{P}$ be the distribution of a stationary IDp process
with L\'{e}vy measure $Q$ and $((\mathbb{R}^{\mathbb{Z}})^{*},(\mathcal
{B}^{\otimes\mathbb{Z}})^{*},Q^{*},S_{*})$
the Poisson measure over the L\'{e}vy measure system $(\mathbb
{R}^{\mathbb{Z}},\mathcal{B}^{\otimes\mathbb{Z}},Q,S)$.
Set $X_{0}$ as\vadjust{\goodbreak} $\{ x_{n}\} _{n\in\mathbb{Z}}\mapsto x_{0}$
and define, on $(\mathbb{R}^{\mathbb{Z}})^{*}$, the stochastic
integral $I(X_{0})$ as the limit in probability, as $n$
tends to infinity, of the random variables
\[
\nu\mapsto{ \int_{|X_{0}|>{1}/{n}}}X_{0}\,\mathrm{d}\nu-{ \int_{|X_{0}|>{1}/{n}}}c(X_{0})\,\mathrm{d}Q.
\]

Then the process $\{ I(X_{0})\circ S_{*}^{n}\} _{n\in\mathbb{Z}}$
has distribution $\mathbb{P}$.
\end{theorem}

\subsection{\texorpdfstring{L\'{e}vy measure as
Maharam system}{Levy measure as
Maharam system}}\label{subL=0000E9vy-measure-Mahram}
\begin{theorem}
\label{thmspecrep}Let $(\mathbb{R}^{\mathbb{Z}},\mathcal{B}^{\otimes
\mathbb{Z}},Q,S)$
be the L\'{e}vy measure system of an $\alpha$-stable stationary process.
Then there exists a~probability space $(\Omega,\mathcal{F},\mu)$,
a~nonsingular transformation $T$, a function $f\in L^{\alpha}(\mu)$
such that, if $M$ denotes the map $(\omega,t)\mapsto tf(\omega)$,
then the map $\Theta:=(\omega,t)\mapsto\{ M\circ\widetilde{T}_{\alpha
}^{n}(\omega,t)\} _{n\in\mathbb{Z}}$
yields an isomorphism of the Maharam system $(\Omega\times\mathbb
{R}_{+},\mathcal{F}\otimes\mathcal{B}_{+},\mu\otimes\frac{1}{s^{1+\alpha
}}\,\mathrm{d}s,\allowbreak\widetilde{T}_{\alpha})$
with $(\mathbb{R}^{\mathbb{Z}},\mathcal{B}^{\otimes\mathbb{Z}},Q,S)$.
\end{theorem}
\begin{pf}
First observe that Theorem \ref{thmMaharam} can be applied to $(\mathbb
{R}^{\mathbb{Z}},\mathcal{B}^{\otimes\mathbb{Z}},Q,S)$
since the measurable and (obviously) dissipative flow $\{ R_{\mathrm
{e}^{{t}/{\alpha}}}\} _{t\in\mathbb{R}}$
satisfies the hypothesis, thanks to equation (\ref{eqStable}). Therefore,
there exists an isomorphism $\Psi$ between the Maharam system $(\Omega
\times\mathbb{R}_{+},\mathcal{F}\otimes\mathcal{B}_{+},\mu\otimes\frac
{1}{s^{1+\alpha}}\,\mathrm{d}s,\widetilde{T}_{\alpha})$
and $(\mathbb{R}^{\mathbb{Z}},\mathcal{B}^{\otimes\mathbb{Z}},Q,S)$
for an appropriate nonsingular system $(\Omega,\mathcal{F},\mu,T)$.
Set $f:=\Psi(\omega,1)_{0}$ [i.e., $\Psi(\omega,1)_{0}$
is the $0$th coordinate of the sequence $\Psi(\omega,1)$],
and let us check that $f\in L^{\alpha}(\mu)$. Indeed,
as $Q$ is a L\'{e}vy measure, we have
\[
\int_{\mathbb{R}^{\mathbb{Z}}}x_{0}^{2}\wedge1Q(\mathrm{d}\{ x_{n}\}
_{n\in\mathbb{Z}})<\infty,
\]
but since $\Psi$ is an isomorphism and $\Psi(\omega,t)=\Psi\circ
S_{t}(\omega,1)=R_{t}\circ\Psi(\omega,1)=t\Psi(\omega,1)$,
we have
\begin{eqnarray*}
&&\int_{\mathbb{R}^{\mathbb{Z}}}x_{0}^{2}\wedge1Q(\mathrm{d}\{ x_{n}\}
_{n\in\mathbb{Z}})\\
&&\qquad  =  \int_{\Omega}\int_{\mathbb{R}_{+}}\Psi(\omega
,t)_{0}^{2}\wedge1\frac{1}{t^{\alpha+1}}\,\mathrm{d}t\mu(\mathrm{d}\omega),\\
&&\int_{\Omega}\int_{\mathbb{R}_{+}}(t^{2}\Psi(\omega,1)_{0}^{2})\wedge
1\frac{1}{t^{\alpha+1}}\,\mathrm{d}t\mu(\mathrm{d}\omega) \\
&&\qquad =  \biggl(\int_{\mathbb
{R}_{+}}z^{2}\wedge1\frac{1}{z^{\alpha+1}}\,\mathrm{d}z\biggr)\int_{\Omega}|\Psi
(\omega,1)_{0}|^{\alpha}\mu(\mathrm{d}\omega)
\end{eqnarray*}
after the change of variable $z:=t|\Psi(\omega,1)_{0}|$.
Therefore $\int_{\Omega}|\Psi(\omega,1)_{0}|^{\alpha}\mu(\mathrm{d}\omega
)<\infty$.
\end{pf}

In the symmetric case we can make the theorem more precise:
\begin{theorem}
\label{thmsymspecrep}Let $(\mathbb{R}^{\mathbb{Z}},\mathcal{B}^{\otimes
\mathbb{Z}},Q,S)$
be the L\'{e}vy measure system of a symmetric $\alpha$-stable
stationary\vadjust{\goodbreak}
process. Then there exists a probability space $(X,\mathcal{A},\nu)$,
a nonsingular transformation $R$, a function $f\in L^{\alpha}(\nu)$
and a measurable map $\xi\dvtx X\to\{ -1,1\} $ such that, if
$M$ denotes the map $(x,t)\mapsto tf(x)$, then
the map $(x,t)\mapsto\{ M\circ\overline{R_{\alpha}}^{n}(x,t)\} _{n\in
\mathbb{Z}}$
yields an isomorphism between $(X\times\mathbb{R}^{*},\mathcal{A}\otimes
\mathcal{B},\nu\otimes\frac{1}{|s|^{1+\alpha}}\,\mathrm{d}s,\overline
{T_{\alpha}})$
with $(\mathbb{R}^{\mathbb{Z}},\mathcal{B}^{\otimes\mathbb{Z}},Q,S)$,
$\overline{R_{\alpha}}$ being defined by $(x,t)\mapsto(Rx,\xi(x)t(\frac
{\mathrm{d}R_{*}^{-1}\mu}{\mathrm{d}\mu}(x))^{{1}/{\alpha}})$.
\end{theorem}
\begin{pf}
Start by applying Theorem \ref{thmspecrep} to the L\'{e}vy measure.

Observe that the symmetry involves the presence of a measure preserving
involution $I$, namely $I\{ x_{n}\} _{n\in\mathbb{Z}}=\{ -x_{n}\} _{n\in
\mathbb{Z}}$.
$I$ also preserves the L\'{e}vy measure of the process. Observe also
that $I$ commutes with the shift and with the flow $R_{t}$. Therefore
$\widetilde{I}:=\Theta^{-1}I\Theta$ is a measure preserving automorphism
of $(\Omega\times\mathbb{R}_{+}^{*},\mathcal{F}\otimes\mathcal{B},\mu
\otimes\frac{1}{s^{1+\alpha}}\,\mathrm{d}s,\widetilde{T}),$
and we can apply Lemma \ref{lemTwoMaharam} to deduce that $\widetilde{I}$
induces a nonsingular involution $\phi$ on $(\Omega,\mathcal{F},\mu,T)$.
It is standard that such transformation admits an equivalent finite
invariant measure, so, up to another measure preserving isomorphism,
we can assume that $\phi$ preserves the probability measure $\mu$.

Using Rohklin's structure theorem, the compact factor associated to
the compact group $\{ \operatorname{Id},\phi\} $ tells us that
we can represent $(\Omega,\mathcal{F},\mu,T)$ as $(X\times\{ -1,1\}
,\mathcal{A}\otimes\mathcal{P}\{ -1,1\} ,\nu\otimes m,R_{\xi})$,
where $R$ is a nonsingular automorphism of $(X,\mathcal{A},\nu)$,
$m$ is the uniform measure on $(\{ -1,1\} ,\mathcal{P}\{ -1,1\} )$,
$\xi$ a cocycle from $X$ to $\{ -1,1\} $ and $R_{\xi}:=(x,\varepsilon
)\mapsto(Rx,\xi(x)\varepsilon)$.

It is now clear that $(X\times\{ -1,1\} \times\mathbb
{R}_{+}^{*},(\mathcal{A}\otimes\mathcal{P}\{ -1,1\} )\otimes\mathcal
{B},\nu\otimes m\otimes\frac{1}{s^{1+\alpha}}\,\mathrm{d}s,\widetilde
{S_{\xi}})$
is isomorphic to $(X\times\mathbb{R}^{*},\mathcal{A}\otimes\mathcal
{B},\nu\otimes\frac{1}{|s|^{1+\alpha}}\,\mathrm{d}s,\overline{R_{\alpha}})$
thanks to the mapping $(x,\varepsilon,t)\mapsto(x,2^{{1}/{\alpha
}}\varepsilon t)$
and $\overline{R_{\alpha}}:=(x,t)\mapsto(Rx,\xi(x)(\frac{\mathrm
{d}R_{*}^{-1}\mu}{\mathrm{d}\mu}(x))^{{1}/{\alpha}}t)$.\vspace*{-3pt}
\end{pf}

\subsection{Spectral representation}

It is now very easy to derive spectral representations from the above
results. In particular, if $(\mathbb{R}^{\mathbb{Z}},\mathcal
{B}^{\otimes\mathbb{Z}},Q,S)$
is the L\'{e}vy measure system of an $S\alpha S$ stationary process,
under the notation of Theorem \ref{thmsymspecrep}, $(X,\mathcal{A},\nu)$
together with the function $f\in L^{\alpha}(\nu)$, the
cocycle $\phi$ and the nonsingular automorphism $T$ yields a spectral
representation of the process. Indeed, by building the Poisson measure
over $(X\times\mathbb{R}^{*},\mathcal{A}\otimes\mathcal{B},\nu\otimes
\frac{1}{|s|^{1+\alpha}}\,\mathrm{d}s,\overline{R_{\alpha}})$
and by applying to it, $f$ and $\xi$ (Theorem 3.12.2, page 156 in
\cite{SamTaq94Stable}), we recover the $S\alpha S$ process with L\'{e}vy
measure $Q$, which proves the validity of the spectral representation.
The minimality can be obtained without difficulty thanks to Proposition
2.2 in \cite{Ros06Minirep}.

\subsection{\texorpdfstring{Maharam systems as L\'{e}vy measure}{Maharam systems as Levy measure}}\label{subMaharam-L=0000E9vy}

We can ask if whether a Maharam system $(\Omega\times\mathbb
{R}_{+},\mathcal{F}\otimes\mathcal{B}_{+},\mu\otimes\frac{1}{s^{1+\alpha
}}\,\mathrm{d}s,\widetilde{T}_{\alpha})$
can be coded into a L\'{e}vy measure system of a stable process. We can
answer this question affirmatively in the only interesting case, that
is, when the Maharam system has no finite absolutely continuous
$\widetilde{T}_{\alpha}$-invariant
measure, that is, when the resulting L\'{e}vy measure system leads to
an ergodic stable process.\vadjust{\goodbreak}

Recall that a Maharam system $(\Omega\times\mathbb{R}_{+},\mathcal
{F}\otimes\mathcal{B}_{+},\mu\otimes\frac{1}{s^{1+\alpha}}\,\mathrm
{d}s,\widetilde{T}_{\alpha})$
has no finite absolutely continuous $\widetilde{T}_{\alpha}$-invariant
measure if and only if the nonsingular system $(\Omega,\mathcal{F},\mu,T)$
has the same property. But then a famous theorem of Krengel \cite{Kren70Gen}
shows that such a system possesses a $2$-generator, that is, there
exists a measurable function $f\dvtx \Omega\to\{ 0,1\} $, $\mu\{ f=1\}
<\infty$,
such that $\sigma\{ f\circ T^{n}, n\in\mathbb{Z}\} =\mathcal{F}$.

To be more precise, this means that, up to isomorphism, $(\Omega
,\mathcal{F},\mu,T)$
can be represented as $(\{ 0,1\} ^{\mathbb{Z}},\mathcal{B}(\{ 0,1\}
^{\mathbb{Z}}),\nu,S)$
for an appropriate measure $\nu$, where $S$ is the shift transformation.
Then the Maharam system can be represented as $(\{ 0,1\} ^{\mathbb
{Z}}\times\mathbb{R}_{+}^{*},\mathcal{B}(\{ 0,1\} ^{\mathbb{Z}})\otimes
\mathcal{B},\nu\otimes\frac{1}{s^{1+\alpha}}\,\mathrm{d}s,\widetilde
{S}_{\alpha})$.
But if $\varphi$ is the map $(\{ x_{n}\} _{n\in\mathbb{Z}},t)\mapsto\{
tx_{n}\} _{n\in\mathbb{Z}}$
and $Q=\varphi_{*}(\nu\otimes\frac{1}{s^{1+\alpha}}\,\mathrm{d}s)$,
we obtain a L\'{e}vy measure system $(\mathbb{R}^{\mathbb{Z}},\mathcal
{B}^{\otimes\mathbb{Z}}, Q,S)$
of an $\alpha$-stable system, as $\{ x_{n}\} _{n\in\mathbb{Z}}\mapsto\{
x_{0}\} $
is in $L^{\alpha}(\mu)$ (see the proof of Theorem \ref{thmspecrep}).
Moreover, the sequence $\{ y_{n}\} _{n\in\mathbb{Z}}$
takes only two values, $0$ or $\sup\{ y_{n}\} _{n\in\mathbb{Z}}$
$Q$-almost everywhere ($\{ y_{n}\} _{n\in\mathbb{Z}}$
can not be identically zero with positive measure as such a constant
sequence forms a shift-invariant set of finite measure); therefore,
$\varphi^{-1}$ exists and is defined by $\{ y_{n}\} _{n\in\mathbb
{Z}}\mapsto(\sup\{ y_{n}\} _{n\in\mathbb{Z}},\{ \frac{y_{n}}{\sup\{
y_{n}\} _{n\in\mathbb{Z}}}\} _{n\in\mathbb{Z}})$.\vspace*{-3pt}

$(\{ 0,1\} ^{\mathbb{Z}}\!\times\!\mathbb{R}_{+}^{*},\mathcal{B}(\{ 0,1\}
^{\mathbb{Z}})\!\otimes\!\mathcal{B},\nu\!\otimes\!\frac{1}{s^{1\!+\!\alpha}}\,\mathrm
{d}s,\widetilde{S}_{\alpha})$
is isomorphic to $(\mathbb{R}^{\mathbb{Z}},\mathcal{B}^{\otimes\mathbb
{Z}}, Q,S)$.

\section{Refinements of the representation}\label{secRefinements-of-the}

Ergodic stationary processes are building blocks of stationary processes;
prime numbers are the building blocks of integers; factors are building
blocks of Von Neumann algebras etc. What are the building blocks of
stationary infinitely divisible processes? Let us get more
precise.

Given a stationary ID process $X$, what are the solutions to the
equation (in distribution)
\[
X=X_{1}+X_{2},
\]
where $X_{1}$ and $X_{2}$ are independent stationary ID processes.
Of course, if $Q$ is the L\'{e}vy measure of $X$, then taking $X_{1}$
with L\'{e}vy measure $c_{1}Q$ and $X_{2}$ with L\'{e}vy measure $c_{2}Q$
with $c_{1}+c_{2}=1$ gives a solution. If these are the only solutions,
we said in \cite{Roy05these} that $X$ is \textit{pure}, meaning that
is impossible to reduce $X$ to ``simpler'' pieces. It was then
very easy to show that $X$ is pure if and only if its L\'{e}vy measure
is ergodic.\vspace*{-3pt}
\begin{prop}
A stationary IDp process $X$ is pure if and only if its L\'{e}vy measure
$Q$ is ergodic.\vspace*{-3pt}
\end{prop}
\begin{pf}
Assume Q is not ergodic. There exists a partition of $\mathbb
{R}^{\mathbb{Z}}$
into two shift invariant sets $A$ and $B$ both of positive measure.
Therefore, $Q_{\mid A}$ and $Q_{\mid B}$ can be taken as L\'{e}vy measures
of two stationary IDp processes $X_{A}$ and $X_{B}$ and taking them
independently leads to
\[
X=X_{A}+X_{B}
\]
in distribution, as $Q=Q_{\mid A}+Q_{\mid B}$.\vadjust{\goodbreak}

In the converse, assume $Q$ is ergodic, and suppose there exist independent
stationary IDp processes $X_{1}$ and $X_{2}$ with L\'{e}vy measure $Q_{1}$
and $Q_{2}$ such that
\[
X=X_{1}+X_{2}
\]
holds in distribution. As $Q=Q_{1}+Q_{2}$ we get $Q_{1}\ll Q$. But
as $Q$ is ergodic, this in turns implies that there exists $c>0$
such that $Q_{1}=cQ$ and thus $Q_{2}=(1-c)Q$.
\end{pf}

In this section, we will try to comment the above equation according
to the Krieger type of the associated nonsingular transformation.
A description of the interesting class of those stable processes driven
by nonsingular transformations of type $\mathrm{III}_{0}$ is unknown.

\subsection{The type $\mathrm{III}_{1}$ case, pure stable processes}

It was an open question whether there exist pure stable processes.
It can now be solved thanks to the Maharam structure of the L\'{e}vy measure:
an $\alpha$-stable process is pure if and only if the underlying
nonsingular system is of type $\mathrm{III}_{1}$.

The existence of pure stable processes (guaranteed by the comments
made in Section \ref{subMaharam-L=0000E9vy}) is reassuring as it
validates the specific study of stable processes.

\subsection{\texorpdfstring{The type $\mathrm{III}_{\lambda}$ case, $0<\lambda<1$}{The type III_lambda case, 0<lambda<1}}

In this section we derive the form of those $\alpha$-stable processes
associated with an ergodic, type $\mathrm{III}_{\lambda}$ nonsingular
automorphism, $0<\lambda<1$.

\subsubsection{Semi-stable stationary processes}

An infinitely divisible probability measure $\mu$ on $\mathbb{R}^{d}$
is called \textit{$\alpha$-semi-stable} with \textit{span $b$} if its
Fourier transform satisfies
\[
\hat{\mu}(z)^{b^{\alpha}}=\hat{\mu}(bz)\mathrm{e}^{i\langle c,z\rangle}
\]
for some $c\in\mathbb{R}^{d}$.

By extension, an $\alpha$-semi-stable process process is a process
whose finite-dimensional distributions are $\alpha$-semi-stable.
Using once again results of Chapter~3 in \cite{Sato99LevPro}, one
gets the following characterization of $\alpha$-semi-stable stationary
processes:

A shift-invariant L\'{e}vy measure $Q$ on $(\mathbb{R}^{\mathbb
{Z}},\mathcal{B}^{\otimes\mathbb{Z}})$
is the L\'{e}vy measure of an $\alpha$-semi-stable stationary process
of span $b>0$ if and only if it satisfies
\[
(R_{b})_{*}Q=b^{-\alpha}Q,
\]
where $R_{b}$ is the multiplication by $b$
\[
\{ x_{n}\} _{\in\mathbb{Z}}\mapsto\{ bx_{n}\} _{\in\mathbb{Z}}.
\]
Of course by iterating $R_{b}$, we easily observe that
$(R_{b^{n}})_{*}Q=b^{-n\alpha}Q$
for all $n\in\mathbb{Z}$.

\subsubsection{Discrete Maharam extension}

Assume $(\Omega,\mathcal{F},\mu,T)$ is a nonsingular
system such that there exists $\lambda>0$ so that $\frac{\mathrm
{d}T_{*}^{-1}\mu}{\mathrm{d}\mu}\in\{ \lambda^{n}, n\in\mathbb{Z}\} $
$\mu$-almost everywhere. We can form its \textit{discrete Maharam extension},
that is, the m.p. system $(\Omega\times\mathbb{Z},\mathcal{F}\otimes
\mathcal{B},\mu\otimes\lambda^{n}\,\mathrm{d}n,\widetilde{T})$
where $\lambda^{n}\,\mathrm{d}n$ stands for the measure ${ \sum_{n\in\mathbb
{Z}}}\lambda^{n}\delta_{n}$
on $(\mathbb{Z},\mathcal{B})$ and $\widetilde{T}$ is
defined by
\[
\widetilde{T}(\omega,n)=\biggl(T\omega,n-\log_{\lambda}\frac{\mathrm
{d}T_{*}^{-1}\mu}{\mathrm{d}\mu}(\omega)\biggr).
\]

\subsubsection{\texorpdfstring{Ergodic decomposition of Maharam extension of type
$\mathrm{III}_{\lambda}$
transformations}{Ergodic decomposition of Maharam extension of type
III_lambda
transformations}}

Let $(\Omega,\mathcal{F},\mu,T)$ be an ergodic type $\mathrm
{III}_{\lambda}$
system. Up to a change of measure we can assume that the Radon--Nykodim
derivative take its values in the group $\{ \lambda^{n}, n\in\mathbb
{Z}\} $
where $r(T)=\{ 0,\lambda^{n}, n\in\mathbb{Z},\mathrm{+\ensuremath{\infty
}}\} $
is the ratio set of~$T$ (see \cite{KatWei91nonsing}). Therefore,
the discrete Maharam extension $(\Omega\times\mathbb{Z},\mathcal
{F}\otimes\mathcal{B},\beta\mu\otimes\lambda^{n}\,\mathrm{d}n,\widetilde{T})$,
where $\beta:=\int_{0}^{-\ln\lambda}\mathrm{e}^{-s}\,\mathrm{d}s$, exists.
Now form the product system
\[
\biggl(\Omega\times\mathbb{Z}\times[0,-\ln\lambda[,\mathcal{F}\otimes\mathcal
{B}\otimes\mathcal{B}([0,-\ln\lambda[),\beta\mu\otimes\lambda^{n}\,\mathrm
{d}n\otimes\frac{\mathrm{e}^{-s}}{\beta}\mathrm{ds},\widetilde{T}\times\operatorname{Id}\biggr).
\]
The dissipative nonsingular flow $S_{t}\dvtx (\omega,n,s)\mapsto(\omega
,n+\lfloor\frac{s-t}{-\ln\lambda}\rfloor,s-t+\break \ln\lambda\lfloor\frac
{s-t}{-\ln\lambda}\rfloor)$
satisfies $S_{t}\circ\widetilde{T}\times\operatorname{Id}=\widetilde
{T}\times\operatorname{Id}\circ S_{t}$
and $(S_{t})_{*}\mu\otimes\lambda^{n}\,\mathrm{d}n\otimes\mathrm{e}^{-s}\,\mathrm
{ds}=\mathrm{e}^{-t}\mu\otimes\lambda^{n}\,\mathrm{d}n\otimes\,\mathrm
{e}^{-s}\mathrm{ds}$,
and it is very easy to see that $(\mathbb{Z}\times[0,-\ln\lambda
[,\mathcal{B}\otimes\mathcal{B}([0,-\ln\lambda[),\lambda^{n}\,\mathrm
{d}n\otimes\mathrm{e}^{-s}\,\mathrm{ds})$
is just a reparametrization of $(\mathbb{R},\mathcal{B},\mathrm
{e}^{s}\,\mathrm{d}s)$,
thanks to the mapping $(n,s)\mapsto-n\ln\lambda-s$.

Therefore $(\Omega\times\mathbb{Z}\times[0,-\ln\lambda[,\mathcal
{F}\otimes\mathcal{B}\otimes\mathcal{B}([0,-\ln\lambda[),\mu\otimes
\lambda^{n}\,\mathrm{d}n\otimes\mathrm{e}^{-s}\,\mathrm{ds},\widetilde{T}\times
\operatorname{Id})$
can be seen as the Maharam extension of $(\Omega,\mathcal{F},\mu,T)$.

It remains to prove the ergodicity of $(\Omega\times\mathbb{Z},\mathcal
{F}\otimes\mathcal{B},\beta\mu\otimes\lambda^{n}\,\mathrm{d}n,\widetilde{T})$.
This follows, for example, from Corollary 5.4 in \cite{Schmidt77Cocycles},
as the ratio set is precisely the set of essential values of the Radon--Nykodim
cocycle.

We then obtain the ergodic decomposition of the Maharam extension:
it is the discrete Maharam extension $(\Omega\times\mathbb{Z},\mathcal
{F}\otimes\mathcal{B},\beta\mu\otimes\lambda^{n}\,\mathrm{d}n,\widetilde{T})$
randomized by the measure $\frac{\mathrm{e}^{-s}}{\beta}\,\mathrm{ds}$
on $[0,-\ln\lambda[$.

\subsubsection{Application to stable processes}

Let $(\mathbb{R}^{\mathbb{Z}},\mathcal{B}^{\otimes\mathbb{Z}},Q,S)$
be the L\'{e}vy measure system of an $\alpha$-stable process driven by
an ergodic type $\mathrm{III}_{\lambda}$ system $(\Omega,\mathcal{F},\mu,T)$,
and let $f\in L^{\alpha}(\mu)$ be given as in Theorem
\ref{thmspecrep}. Let $b>1$ so that $b^{-\alpha}=\lambda$, we
need to obtain a multiplicative version of the above structure adapted
to our parameters. Up to a change of measure we can assume that $(\frac
{\mathrm{d}T_{*}^{-1}\mu}{\mathrm{d}\mu})^{{1}/{\alpha}}\in\{ b^{n},
n\in\mathbb{Z}\} $
$\mu$-almost everywhere. Consider the discrete Maharam extension
$(\Omega\times G_{b},\mathcal{F}\otimes\mathcal{B},\beta\mu\otimes
m_{b},\widetilde{T})$
(in a multiplicative representation) where $\beta=\int_{1}^{b}\frac
{1}{s^{1+\alpha}}\,\mathrm{d}s$,
$m_{b}$ is the measure ${ \sum_{g\in G_{b}}}g^{-\alpha}\delta_{g}$
on $G_{b}:=\{ b^{n}, n\in\mathbb{Z}\} $ and\vadjust{\goodbreak} $\widetilde{T}:=(\omega
,g)\mapsto(T\omega,g(\frac{\mathrm{d}T_{*}^{-1}\mu}{\mathrm{d}\mu}(\omega
))^{{1}/{\alpha}})$.
Form\vfill\eject the system $(\mathbb{R}^{\mathbb{Z}},\mathcal{B}^{\otimes\mathbb
{Z}},Q^{r},S)$
as a factor of $(\Omega\times G_{b},\mathcal{F}\otimes\mathcal{B},\beta
\mu\otimes m_{b},\widetilde{T})$
given by the mapping $\varphi:=(\omega,g)\mapsto\{ M\circ\widetilde
{T}^{n}(\omega,g)\} _{n\in\mathbb{Z}}$
where $M(\omega,g)=gf(\omega)$ and $Q^{r}=\varphi_{*}(\beta\mu\otimes m_{b})$.

Now, as above, we recover the Maharam extension of $(\Omega,\mathcal
{F},\mu,T)$
by considering $(\Omega\times G_{b}\times[1,b[,\mathcal{F}\otimes
\mathcal{B}\otimes\mathcal{B}([1,b[),\beta\mu\otimes m_{b}\otimes\frac
{1}{\beta s^{1+\alpha}}\,\mathrm{d}s,\widetilde{T}\times\operatorname{Id})$.
As the system $(G_{b}\times[1,b[,\mathcal{B}\otimes\mathcal
{B}([1,b[),m_{b}\otimes\frac{1}{s^{1+\alpha}}\,\mathrm{d}s)$
is isomorphic to $(\mathbb{R}_{+}^{*},\mathcal{B},\frac{1}{s^{1+\alpha
}}\,\mathrm{d}s)$
thanks to $(g,t)\mapsto gt$, we obtain $(\mathbb{R}^{\mathbb
{Z}},\mathcal{B}^{\otimes\mathbb{Z}},Q,S)$
by applying the map\break $(\{ x_{n}\} _{n\in\mathbb{Z}},t)\mapsto\{ tx_{n}\}
_{n\in\mathbb{Z}}$
to $(\mathbb{R}^{\mathbb{Z}}\times[1,b[,\mathcal{B}^{\otimes\mathbb
{Z}}\otimes\mathcal{B}([1,b[),Q^{r}\otimes\frac{1}{\beta s^{1+\alpha
}}\,\mathrm{d}s,S\times\operatorname{Id})$.
At last, we can check that $Q^{r}$ is a L\'{e}vy measure, and indeed we know
that
\[
\int_{\mathbb{R}^{\mathbb{Z}}}(x_{0}^{2}\wedge1)Q(\mathrm{d}\{ x_{n}\}
_{n\in\mathbb{Z}})<+\infty,
\]
but
\[
\int_{\mathbb{R}^{\mathbb{Z}}}(x_{0}^{2}\wedge1)Q(\mathrm{d}\{ x_{n}\}
_{n\in\mathbb{Z}})=\int_{1}^{b}\biggl(\int_{\mathbb{R}^{\mathbb
{Z}}}\bigl((sx_{0})^{2}\wedge1\bigr)Q^{r}(\mathrm{d}\{ x_{n}\} _{n\in\mathbb
{Z}})\biggr)\frac{1}{\beta s^{1+\alpha}}\,\mathrm{d}s;
\]
therefore, for some $1\leq s<b$, $\int_{\mathbb{R}^{\mathbb
{Z}}}((sx_{0})^{2}\wedge1)Q^{r}(\mathrm{d}\{ x_{n}\} _{n\in\mathbb
{Z}})<+\infty,$
and this is enough to prove that $Q^{r}$ is a L\'{e}vy measure.

$(\mathbb{R}^{\mathbb{Z}},\mathcal{B}^{\otimes\mathbb{Z}},Q^{r},S)$
is the L\'{e}vy measure system of an $\alpha$-semi-stable stationary
process with span $b$. Heuristically, if $X$ has L\'{e}vy measure $Q$,
$X$ can be thought of as the continuous sum of independent processes
$Y^{t}$, $1\leq t<b$ weighted by the probability measure $\frac
{1}{\beta s^{1+\alpha}}\,\mathrm{d}s$
where $\frac{1}{t}Y^{t}$ has L\'{e}vy measure $Q^{r}$. More formally,
if $((\Omega\times G_{b}\times[1,b[)^{*},(\mathcal{F}\otimes\mathcal
{B}\otimes\mathcal{B}([1,b[))^{*},(\beta\mu\otimes m_{b}\otimes\frac
{1}{\beta s^{1+\alpha}}\,\mathrm{d}s)^{*},(\widetilde{T}\times\operatorname{Id})_{*})$
denotes the Poisson suspension over $(\Omega\times G_{b}\times
[1,b[,\mathcal{F}\otimes\mathcal{B}\otimes\mathcal{B}([1,b[),\beta\mu
\otimes m_{b}\otimes\frac{1}{\beta s^{1+\alpha}}\,\mathrm{d}s,\widetilde
{T}\times\operatorname{Id})$,
then, if $I$ denotes the stochastic integral as in Theorem~\ref{thmMaruyama-rep},
$X:=\{ I\{ M_{1}\} \circ(\widetilde{T}\times\operatorname{Id})_{*}^{n}\} _{n\in
\mathbb{Z}}$
has L\'{e}vy measure $Q$ and $Y:=\{ I\{ M_{2}\} \circ(\widetilde
{T}\times\operatorname{Id})_{*}^{n}\} _{n\in\mathbb{Z}}$
where $M_{1}(\omega,g,s)=sgf(\omega)$ and $M_{2}(\omega,g,s)=gf(\omega)$.

We therefore observe that $X$ is entirely determined by a pure $\alpha
$-semi-stable
stationary process with span $b$, $Y$. It is very easy to see that
$X$ and~$Y$ share the same type of mixing.\vspace*{-3pt}

\subsubsection{Examples}

It is not difficult to exhibit examples of stable processes of the
kind described above, as the structure detailed allows us to build such
processes. We can, for example, consider the systems $T_{p}$, $\frac{1}{2}<p<1$
introduced in~\cite{HajKakIt72invmeas}. We will follow the presentation
given in (\cite{Aar97InfErg}, page 104).

Let $\Omega$ be the group of dyadic integers, let $\tau$ acts by
translation by $\underline{1}$ on $\Omega$ and for $\frac{1}{2}\!<\!p\!<\!1$,
let $\mu_{p}$ be a probability measure on $\Omega$ defined on cylinders~by
\[
\mu_{p}([\varepsilon_{1},\ldots,\varepsilon_{n}])=\prod_{k=1}^{n}p(\varepsilon_{k}),
\]
where $p(0)=1-p$ and $p(1)=p$.\vadjust{\goodbreak}

If we set $\frac{1-p}{p}=\lambda$, we get
\[
\frac{\mathrm{d}\tau_{*}^{-1}\mu_{p}}{\mathrm{d}\mu_{p}}=\lambda^{\phi},
\]
where $\phi(x)=\min\{ n\in\mathbb{N}, x_{n}=0\} -2$.
It is proved in \cite{HajKakIt72invmeas} that the discrete Maharam
extension $(\Omega\otimes\mathbb{Z},\mathcal{F}\otimes\mathcal{B},\mu
\otimes\lambda^{n}\,\mathrm{d}n,\widetilde{\tau})$
is ergodic.

We can form a new system, which will be the L\'{e}vy measure system of
a~stationary semi-stable process with span $\lambda^{\alpha}$, thanks
to the following map:
\[
f\dvtx (\omega,n)\mapsto\lambda^{\alpha n}\sum_{i\ge1}\omega_{i}2^{-i}.
\]

The L\'{e}vy measure $Q^{r}$ is the image of $\mu\otimes\lambda^{n}\,\mathrm{d}n$
by the map $(\omega,n)\mapsto\{ f\circ\widetilde{\tau}^{k}(\omega,n)\}
_{k\in\mathbb{Z}}$.

By randomizing this L\'{e}vy measure as explained above, we obtain the
L\'{e}vy measure $Q$ of a stationary $\alpha$-stable process; that is,
$Q$ is the image measure of $Q^{r}\otimes\frac{1}{\beta s^{1+\alpha
}}\,\mathrm{d}s$
by the map
\[
(\{ x_{n}\} _{n\in\mathbb{Z}},t)\mapsto\{ tx_{n}\} _{n\in\mathbb{Z}}.
\]

To obtain a realization of these two processes as stochastic integrals
over Poisson suspensions, we can proceed as explained at the end of
the preceding section.

Anticipating the next sections, we derive the ergodic properties of these
processes:

$\tau$ being of type $\mathrm{III}_{\lambda}$, the Maharam extension
is of type $\mathrm{II}_{\infty}$ which means that the L\'{e}vy measure
system of the corresponding stationary $\alpha$-stable process (with
L\'{e}vy measure $Q$) is of type $\mathrm{II}_{\infty}$. Therefore the
associated Poisson suspension is weakly mixing. As stochastic integrals
with respect to this Poisson suspension, both processes (with L\'{e}vy
measures $Q$ and $Q^{r}$) are weakly mixing.

Thanks to (Lemma 1.2.10, page 30 in \cite{Aar97InfErg}), $\tau$ is
rigid for the sequence $\{ 2^{n}\} _{n\in\mathbb{N}}$.
Therefore, by Theorem \ref{thmRigidStable} (or with a slight adaptation
for the semi-stable case), both processes are also rigid for the same
sequence.

\subsection{The type $\mathrm{I}$ and $\mathrm{II}$ cases}

This case is easy to deal with as we can assume the associated ergodic
nonsingular system is actually measure preserving, that is, the L\'{e}vy
measure system $(\mathbb{R}^{\mathbb{Z}},\mathcal{B}^{\otimes\mathbb{Z}},Q,S)$
is isomorphic to $(\Omega\times\mathbb{R}_{+}^{*},\mathcal{F}\otimes
\mathcal{B},\mu\otimes\frac{1}{s^{1+\alpha}}\,\mathrm{d}s,\widetilde{T})$
where $T$ preserves $\mu$ and $\widetilde{T}$ acts as $T\times\operatorname{Id}$,
that is, $\widetilde{T}(\omega,t)=(T\omega,t)$.
Considering $f\in L^{\alpha}(\mu)$ furnished by Theorem
\ref{thmspecrep}, $(\Omega\times\mathbb{R}_{+}^{*},\mathcal{F}\otimes
\mathcal{B},\mu\otimes\frac{1}{s^{1+\alpha}}\,\mathrm{d}s,\widetilde{T})$
is isomorphic to $(\mathbb{R}^{\mathbb{Z}}\times\mathbb
{R}_{+}^{*},\mathcal{B}^{\otimes\mathbb{Z}}\otimes\mathcal
{B},Q^{s}\otimes\frac{1}{s^{1+\alpha}}\,\mathrm{d}s,S\times\operatorname{Id})$
through the map $(\omega,t)\mapsto(\{ f\circ T^{n}(\omega)\} _{n\in
\mathbb{Z}},t)$
and
\begin{eqnarray*}
\int_{\mathbb{R}^{\mathbb{Z}}}(x_{0}^{2}\wedge1)Q(\mathrm{d}\{ x_{n}\}
_{n\in\mathbb{Z}}) & = & \int_{\Omega}\int_{\mathbb{R}_{+}}\bigl((tf(\omega
))^{2}\wedge1\bigr)\frac{1}{t^{\alpha+1}}\,\mathrm{d}t\mu(\mathrm{d}\omega)\\
& = & \int_{\mathbb{R}_{+}}\biggl(\int_{\mathbb{R}^{\mathbb
{Z}}}\bigl((tx_{0})^{2}\wedge1\bigr)Q^{s}(\mathrm{d}\{ x_{n}\} _{n\in\mathbb
{Z}})\biggr)\frac{1}{t^{\alpha+1}}\,\mathrm{d}t\\
&<&+\infty.
\end{eqnarray*}

For the same reason as above, $Q^{s}$ is a L\'{e}vy measure. We draw the
same conclusions as in the preceding section, taking into account that
the weight is now the infinite measure $\frac{1}{t^{\alpha+1}}\,\mathrm{d}t$
on $\mathbb{R}_{+}^{*}$, and $Q^{s}$ can be any L\'{e}vy measure (of
a stationary IDp process).

\section{Ergodic properties}\label{secErgodic-properties}

Some ergodic properties of general IDp stationary processes have been
given in terms of ergodic properties of the L\'{e}vy measure system in
\cite{Roy06IDstat}. For an $\alpha$-stable stationary processes,
it is more interesting to give them in terms of the associated nonsingular
system $(\Omega,\mathcal{F},\mu,T)$. This work has been
undertaken in the symmetric ($S\alpha S$) case in a series of papers
(see, in particular, \cite{ros96mixsta} and \cite{Samo04Posnul}).

We have a new tool to deal with this problem:

As the L\'{e}vy measure of an $\alpha$-stable stationary processes can
now be seen as the Maharam extension $(\Omega\times\mathbb
{R}_{+}^{*},\mathcal{F}\otimes\mathcal{B},\mu\otimes\frac{1}{s^{1+\alpha
}}\,\mathrm{d}s,\widetilde{T})$
of the system $(\Omega,\mathcal{F},\mu,T)$, it suffices
to connect ergodic properties of $T$ and $\widetilde{T}$, and then apply
the general results relating ergodic properties of a stationary IDp
process with respect to those of its L\'{e}vy measure system.

Observe that linking ergodic properties of $T$ and $\widetilde{T}$
is a general problem in nonsingular ergodic theory which is of great
interest.

We will illustrate this in the following sections dealing with mixing,
$K$-property and rigidity, the last two having been neglected in
the $\alpha$-stable literature.

\subsection{Mixing}

First recall that if $S$ is a nonsingular transformation of a~measure
space $(X,\mathcal{A},m)$, it induces a unitary operator
$U_{S}$ on $L^{2}(m)$ by
\[
U_{S}f(x)=\sqrt{\frac{\mathrm{d}S_{*}^{-1}\mu}{\mathrm{d}\mu}(x)}f\circ S(x).
\]
We first give a general result:
\begin{prop}
The Maharam system $(\Omega\times\mathbb{R}_{+}^{*},\mathcal{F}\otimes
\mathcal{B}_{+},\mu\otimes\frac{1}{s^{1+\alpha}}\,\mathrm{d}s,\widetilde
{T}_{\alpha})$
is of zero type\textup{\textit{ if and only }}if for all $f\in L^{2}(\mu)$,
$\langle U_{T}^{n}f,f\rangle_{L^{2}(\mu)}\to0$
as $n$ tends to infinity.
\end{prop}

The proof can be extracted from \cite{ros96mixsta} but follows also
from the observation that the conditions below are equivalent (we
assume that $\mu$ is a probability):
\begin{longlist}[(1)]
\item[(1)] for all $f\in L^{2}(\mu)$, $\langle U_{T}^{n}f,f\rangle
_{L^{2}(\mu)}\to0$
as $n$ tends to infinity;
\item[(2)]$|\log\frac{\mathrm{d}T_{*}^{-n}\mu}{\mathrm{d}\mu}|\to\infty$
in probability;\vadjust{\goodbreak}
\item[(3)]$m(A_{\varepsilon}\cap\widetilde{T}_{\alpha}^{n}A_{\varepsilon})\to0$
for all $0<\varepsilon<1$ where $m=\mu\otimes\frac{1}{s^{1+\alpha}}\,\mathrm{d}s$
and $A_{\varepsilon}:=\Omega\times[\varepsilon,\frac{1}{\varepsilon}]$;\vspace*{2pt}
\item[(4)]$\widetilde{T}_{\alpha}$ is of zero type.
\end{longlist}
Combining this result with the characterization of the L\'{e}vy measure
system as a Maharam system and the mixing criteria found in \cite{Roy06IDstat},
we obtain the following theorem, already known in the $S\alpha S$-case
(see \cite{ros96mixsta}):
\begin{theorem}
A stationary stable process $(\mathbb{R}^{\mathbb{Z}},\mathcal
{B}^{\otimes\mathbb{Z}},\mathbb{P},S)$
with associated system $(\Omega,\mathcal{F},\mu,T)$ is
mixing if and only if for all $f\in L^{2}(\mu)$, $\langle
U_{T}^{n}f,\allowbreak f\rangle_{L^{2}(\mu)}\to0$
as $n$ tends to infinity.
\end{theorem}

In the forthcoming sections, we are interested in less-known ergodic
properties ($K$ property and rigidity) that have been neglected in
the $\alpha$-stable literature.

\subsection{$K$ property}
\begin{defn}[(see \cite{SilThieu95skewent})]
\label{defK-system}  A conservative
nonsingular system $(\Omega,\mathcal{F},\mu,T)$ is a
$K$-system if there exists a sub-$\sigma$-algebra $\mathcal{G}\subset
\mathcal{F}$
such that $T^{-1}\mathcal{G}\subset\mathcal{G}$, $T^{-n}\mathcal
{G}\downarrow\{ \Omega,\varnothing\} $,
$T^{n}\mathcal{G}\uparrow\mathcal{F}$ and $\frac{\mathrm{d}\mu}{\mathrm
{d}T_{*}\mu}$
is $\mathcal{G}$-measurable.
\end{defn}

A $K$-system is always ergodic (see \cite{SilThieu95skewent}) .
\begin{defn}
A measure-preserving system $(X,\mathcal{A},m,S)$ is remotely
infinite if there exists a sub-$\sigma$-algebra $\mathcal{C}\subset
\mathcal{A}$
such that $T^{-1}\mathcal{C}\subset\mathcal{C}$, $S^{n}\mathcal
{C}\uparrow\mathcal{A}$
and $\bigcap_{n\ge1}S^{-n}\mathcal{C}$ contains zero or infinite
measure sets only.
\end{defn}
\begin{prop}
\label{proMaharamK}If $(\Omega,\mathcal{F},\mu,T)$ is
a $K$-system which is not of type $\mathrm{II}_{1}$, then its Maharam
extension $(\Omega\times\mathbb{R}_{+}^{*},\mathcal{F}\otimes\mathcal
{B}_{+},\mu\otimes\frac{1}{s^{1+\alpha}}\,\mathrm{d}s,\widetilde
{T}_{\alpha})$
is remotely infinite.
\end{prop}
\begin{pf}
Let $\mathcal{G}$ be as in Definition \ref{defK-system}. Observe
that, as $\frac{\mathrm{d}\mu}{\mathrm{d}T_{*}\mu}$ is $\mathcal{G}$-measur\-able,
$\mathcal{G}\otimes\mathcal{B}_{+}$ is $\widetilde{T}_{\alpha}$-invariant,
that is, $\widetilde{T}_{\alpha}^{-1}\mathcal{G}\,\otimes\,\mathcal
{B}_{+}\subset\mathcal{G}\,\otimes\,\mathcal{B}_{+}$.
Indeed, take $g$ $\mathcal{G}$-measurable and $f$ $\mathcal{B}_{+}$-measurable,
and we get
\begin{eqnarray*}
g\otimes f(\widetilde{T}_{\alpha}(\omega,s))&=&\biggl(g(T\omega),s\biggl(\frac{\mathrm
{d}T_{*}^{-1}\mu}{\mathrm{d}\mu}(\omega)\biggr)^{{1}/{\alpha}}\biggr)\\
&=&\biggl(g(T\omega
),s\biggl(\frac{\mathrm{d}\mu}{\mathrm{d}T_{*}\mu}(T\omega)\biggr)^{{1}/{\alpha}}\biggr).
\end{eqnarray*}

We are going to show that $\mathcal{P}:={ \bigcap_{n\in\mathbb
{N}}}\widetilde{T}_{\alpha}^{-n}\mathcal{G}\otimes\mathcal{B}_{+}$
only contains sets of zero or infinite measure. Observe that, as $S_{t}$
commutes with $\widetilde{T}_{\alpha}$ and preserves $\mathcal{G}\otimes
\mathcal{B}_{+}$
for all $t>0$, then $S_{t}^{-1}(\widetilde{T}_{\alpha}^{-n}\mathcal
{G}\otimes\mathcal{B}_{+})\subset\widetilde{T}_{\alpha}^{-n}\mathcal
{G}\otimes\mathcal{B}_{+}$
and therefore $S_{t}^{-1}\mathcal{P}\subset\mathcal{P}$ for all $t>0$.
Now consider the measurable union, say $K$, of $\mathcal{P}$-measurable
sets of finite and positive measure. It is a $\widetilde{T}_{\alpha}$-invariant
set and a $S_{t}$-invariant set as well. Recall that the nonsingular
action of the flow $S_{t}$ on the ergodic components of $(\Omega\times
\mathbb{R}_{+}^{*},\mathcal{F}\otimes\mathcal{B}_{+},\mu\otimes\frac
{1}{s^{1+\alpha}}\,\mathrm{d}s,\widetilde{T}_{\alpha})$
is ergodic; therefore, if $K\neq\varnothing$, then $K=\Omega\times\mathbb
{R}_{+}^{*}$
mod. $\nu\otimes\frac{1}{s^{1+\alpha}}\,\mathrm{d}s$.

Assume $K=\Omega\times\mathbb{R}_{+}^{*}$. This implies that the
measure $\mu\otimes\frac{1}{s^{1+\alpha}}\,\mathrm{d}s$ is $\sigma$-finite
on $\mathcal{P}$, and therefore $\mathcal{P}$ is a factor of $(\Omega
\times\mathbb{R}_{+}^{*},\mathcal{F}\otimes\mathcal{B}_{+},\mu\otimes
\frac{1}{s^{1+\alpha}}\,\mathrm{d}s,\widetilde{T}_{\alpha})$.
Now consider the quotient space $(\Omega\times\mathbb
{R}_{+}^{*})_{\diagup\mathcal{P}}$
that we can endow, with a~slight abuse of notation with the $\sigma$-algebra
$\mathcal{P}$. Let $\rho$ be the image measure of $\mu\otimes\frac
{1}{s^{1+\alpha}}\,\mathrm{d}s$
by the projection map $\pi$. On $((\Omega\times\mathbb
{R}_{+}^{*})_{\diagup\mathcal{P}},\mathcal{P},\rho)$
$\widetilde{T}_{\alpha}$ and the dissipative flow $S_{t}$ induce
a transformation $U$ and a dissipative flow $Z_{t}$ that satisfy
\[
\pi\circ\widetilde{T}_{\alpha}=U\circ\pi, \pi\circ S_{t}=U\circ\pi\quad\mbox
{and}\quad U\circ Z_{t}=Z_{t}\circ U.
\]

Of course, thanks to Theorem \ref{thmMaharam}, $((\Omega\times\mathbb
{R}_{+}^{*})_{\diagup\mathcal{P}},\mathcal{P},\rho,U)$
is a Maharam system; therefore, we can represent it as $(Y\times\mathbb
{R}_{+}^{*},\mathcal{K}\otimes\mathcal{B}_{+},\sigma\otimes\frac
{1}{s^{1+\alpha}}\,\mathrm{d}s,\widetilde{L}_{\alpha})$
for a nonsingular system $(Y,\mathcal{K},\sigma,L)$.
Applying Lemma \ref{lemTwoMaharam}, $\pi$ induces a nonsingular
factor map $\Gamma$ from $(\Omega,\mathcal{G},\mu,T)$
to $(Y,\mathcal{K},\sigma,L)$, which means that there
exists an $R$-invariant $\sigma$-algebra $\mathcal{Z}\subset\mathcal{G}$
such that $\Gamma^{-1}\mathcal{K}=\mathcal{Z}$. But we can observe,
that for all $n>0$, the factor $\widetilde{T}_{\alpha}^{-n}\mathcal
{G}\otimes\mathcal{B}_{+}$
corresponds to a Maharam system that corresponds to the factor
$T^{-n}\mathcal{G}$
of $(\Omega,\mathcal{G},\mu,T)$. Therefore, for all $n>0$,
$\mathcal{Z}\subset T^{-n}\mathcal{G}$, that is, $\mathcal{Z}\subset{
\bigcap_{n\in\mathbb{N}}}T^{-n}\mathcal{G}=\{ \Omega,\varnothing\} $.
This means that $\mathcal{K}=\{ Y,\varnothing\} $, or, in
other words, that $(Y,\mathcal{K},\sigma,L)$ is the trivial
(one-point) system. $(Y\times\mathbb{R}_{+}^{*},\mathcal{K}\otimes
\mathcal{B}_{+},\sigma\otimes\frac{1}{s^{1+\alpha}}\,\mathrm
{d}s,\widetilde{L}_{\alpha})$
then possesses lots of invariant sets of positive finite measure,
for example $A:=Y\times[1,2]$. But $\pi^{-1}(A)$
is in turn a positive and finite measure invariant set for the system
$(\Omega\times\mathbb{R}_{+}^{*},\mathcal{F}\otimes\mathcal{B}_{+},\mu
\otimes\frac{1}{s^{1+\alpha}}\,\mathrm{d}s,\widetilde{T}_{\alpha})$,
and the existence of such set is impossible in a Maharam extension
of an ergodic system which does not posses a finite $T$-invariant
probability measure $\nu\ll\mu$. We can conclude that $K=\varnothing$.

To prove that $(\Omega\times\mathbb{R}_{+}^{*},\mathcal{F}\otimes
\mathcal{B}_{+},\mu\otimes\frac{1}{s^{1+\alpha}}\,\mathrm{d}s,\widetilde
{T}_{\alpha})$
is remotely infinite, it remains to show that $\bigvee_{n\in\mathbb
{Z}}\widetilde{T}_{\alpha}^{-n}\mathcal{G}\otimes\mathcal
{B}_{+}=\mathcal{F}\otimes\mathcal{B}_{+}$.
We only sketch the proof which consists of verifying that the operation
of taking natural extension and Maharam extension commute.

Of course, we have $\bigvee_{n\in\mathbb{Z}}\widetilde{T}_{\alpha
}^{-n}\mathcal{G}\otimes\mathcal{B}_{+}\subset\mathcal{F}\otimes\mathcal
{B}_{+}$.
It is not difficult to check that $\bigvee_{n\in\mathbb{Z}}\widetilde
{T}_{\alpha}^{-n}\mathcal{G}\otimes\mathcal{B}_{+}$
corresponds to a Maharam system that comes from a $\sigma$-algebra
$\mathcal{H}\subset\mathcal{F}$. But we also have $\mathcal{G}\subset
\mathcal{H}$
and as $T^{-1}\mathcal{H}=\mathcal{H}$, we get $\bigvee_{n\in\mathbb
{Z}}T^{-n}\mathcal{G}\subset\mathcal{H}$.
By assumption, $\bigvee_{n\in\mathbb{Z}}T^{-n}\mathcal{G}=\mathcal{F}$,
and we deduce $\mathcal{H}=\mathcal{F}$ which implies $\bigvee_{n\in
\mathbb{Z}}\widetilde{T}_{\alpha}^{-n}\mathcal{G}\otimes\mathcal
{B}_{+}=\mathcal{F}\otimes\mathcal{B}_{+}$.
\end{pf}

As before we deduce the following result for $\alpha$-stable stationary
processes:

\begin{theorem}
\label{thmStableK}Let $(\mathbb{R}^{\mathbb{Z}},\mathcal{B}^{\otimes
\mathbb{Z}},\mathbb{P},S)$
be a stationary stable process with associated system $(\Omega,\mathcal
{F},\mu,T)$.
If $(\Omega,\mathcal{F},\mu,T)$ is $K$ and not of type
$\mathrm{II}_{1}$, then $(\mathbb{R}^{\mathbb{Z}},\mathcal{B}^{\otimes
\mathbb{Z}},\mathbb{P},S)$
is $K$.
\end{theorem}
\begin{pf}
From Proposition \ref{proMaharamK}, we know that the L\'{e}vy measure
system of the stable process is remotely infinite. The corresponding
Poisson suspension is $K$ by a result from \cite{Roy07Infinite}.
By applying Maruyama's representation Theorem (Theorem \ref{thmMaruyama-rep}),
we recover the stable process as a factor of the suspension, which
therefore inherits the $K$ property.
\end{pf}

Recall that in the probability preserving context, $K$ is strictly
stronger than mixing. In \cite{ros96mixsta}, to produce examples
of mixing $\alpha$-stable stationary processes that were not based
on dissipative nonsingular systems, the authors considered indeed
null recurrent Markov chains as base systems. These systems are
well-known examples of $K$-systems; therefore Theorem \ref{thmStableK}
shows that the associated $\alpha$-stable stationary processes are
not just merely mixing but are indeed $K$.

\subsection{Rigidity}

We recall that a system $(\Omega,\mathcal{F},\mu,T)$ is
rigid if there exists an increasing sequence $n_{k}$ such that
$T^{n_{k}}\rightarrow\operatorname{Id}$
in the group of nonsingular automorphism on $(\Omega,\mathcal{F},\mu)$ [the
convergence being equivalent to the weak convergence in $L^{2}(\mu)$
of the associated unitary operators $U_{T^{n_{k}}}\dvtx f\mapsto\sqrt{\frac
{\mathrm{d}T_{*}^{-n_{k}}\mu}{\mathrm{d}\mu}}f\circ T^{n_{k}}$
to the identity]. Observe that in the finite measure case, rigidity
does not imply ergodicity but prevents mixing.
\begin{prop}
The Maharam system $(\Omega\times\mathbb{R}_{+}^{*},\mathcal{F}\otimes
\mathcal{B}_{+},\mu\otimes\frac{1}{s^{1+\alpha}}\,\mathrm{d}s,\widetilde
{T}_{\alpha})$
is rigid for the sequence $n_{k}$ if and only $(\Omega,\mathcal{F},\mu,T)$
is rigid for the sequence~$n_{k}$.
\end{prop}
\begin{pf}
First observe that the map $T\mapsto\widetilde{T}_{\alpha}$ is a
continuous group homomorphism from the group of nonsingular automorphism
of $(\Omega,\mathcal{F},\mu)$ to the group of measure
preserving automorphism of $(\Omega\times\mathbb{R}_{+}^{*},\mathcal
{F}\otimes\mathcal{B}_{+},\mu\otimes\frac{1}{s^{1+\alpha}}\,\mathrm{d}s)$.
As $T^{n_{k}}\rightarrow\operatorname{Id}$, then $\widetilde{T}_{\alpha
}^{n_{k}}\rightarrow\operatorname{Id}$;
therefore $\widetilde{T}_{\alpha}^{n_{k}}$ is rigid for the sequence
$n_{k}$.\vspace*{2pt}

Conversely, if $\widetilde{T}_{\alpha}$ is rigid for the same sequence,
then, as
\[
\langle U_{\widetilde{T}_{\alpha}}^{n_{k}}f\otimes g,f\otimes g\rangle
_{L^{2}(\mu\otimes{1}/{(s^{1+\alpha})}\,\mathrm{d}s)}\le\Vert g\Vert
_{2}^{2}\langle U_{T}^{n_{k}}f,f\rangle_{L^{2}(\mu)}\leq\Vert g\Vert
_{2}^{2}\Vert f\Vert_{2}^{2}
\]
and $\langle U_{\widetilde{T}_{\alpha}}^{n_{k}}f\otimes g,f\otimes
g\rangle_{L^{2}(\mu\otimes{1}/{(s^{1+\alpha})}\,\mathrm{d}s)}\to\Vert
g\Vert_{2}^{2}\Vert f\Vert_{2}^{2}$,
we get $\langle U_{T}^{n_{k}}f,f\rangle_{L^{2}(\mu)}\to\Vert f\Vert_{2}^{2}$;
thus $T$ is rigid.
\end{pf}

We need the following general result:
\begin{prop}
A stationary IDp stationary process $(\mathbb{R}^{\mathbb{Z}},\mathcal
{B}^{\otimes\mathbb{Z}},\mathbb{P},S)$
is rigid for the sequence $n_{k}$ if and only if its L\'{e}vy measure
system $(\mathbb{R}^{\mathbb{Z}},\mathcal{B}^{\otimes\mathbb{Z}},Q,S)$
is rigid for the sequence $n_{k}$.
\end{prop}
\begin{pf}
Consider $X:=\{ X_{n}\} _{n\in\mathbb{Z}}$ where $X_{n}:=\{ x_{k}\}
_{k\in\mathbb{Z}}\mapsto x_{n}$
on $\mathbb{R}^{\mathbb{Z}}$ and let $\langle a,X\rangle$
be a finite linear combination of the coordinates. $\exp{i}\langle
a,X\rangle-\mathbb{E}[\exp{i}\langle a,X\rangle]$
is a centered square integrable vector under $\mathbb{P}$ whose spectral
measure (under $\mathbb{P}$) is $\lambda_{a}:=|\mathbb{E}[\exp
{i}\langle a,X\rangle]|^{2}{ \sum_{k=1}^{\infty}}\frac{1}{k!}\sigma_{a}^{*k}$
where $\sigma_{a}$ is the spectral measure of $\exp{i}\langle
a,X\rangle-1$
under~$Q$ (see \cite{Roy06IDstat}). Therefore $\widehat{\sigma
_{a}}(n_{k})\to\widehat{\sigma_{a}}(0)$
if and only if $\widehat{\lambda_{a}}(n_{k})\to\widehat{\lambda_{a}}(0)$.
This implies that $\exp{i}\langle a,X\rangle-\mathbb{E}[\exp
{i}\langle a,X\rangle]$
is a rigid vector for $n_{k}$ under $\mathbb{P}$ if and only if
$\exp{i}\langle a,X\rangle-1$ is a rigid vector
for~$n_{k}$ under~$Q$. Observe now that the smallest $\sigma$-algebra
generated by vectors of the kind $\exp{i}\langle a,X\rangle
-\mathbb{E}[\exp{i}\langle a,X\rangle]$
under $\mathbb{P}$ is $\mathcal{B}^{\otimes\mathbb{Z}}$, and the
same is true with vectors of the kind $\exp{i}\langle a,X\rangle-1$
under $Q$. As in any dynamical system if there exists a rigid vector
for the sequence $n_{k}$, there exists a~nontrivial factor which
is rigid for the sequence $n_{k}$, we get the announced result.
\end{pf}
\begin{theorem}
\label{thmRigidStable}A stationary stable process $(\mathbb{R}^{\mathbb
{Z}},\mathcal{B}^{\otimes\mathbb{Z}},\mathbb{P},S)$
with associated system $(\Omega,\mathcal{F},\mu,T)$ is
rigid for the sequence $n_{k}$ if and only $(\Omega,\mathcal{F},\mu,T)$
is rigid for the sequence $n_{k}$.
\end{theorem}
\begin{pf}
This is the combination of the last two results.
\end{pf}
%

%


\printaddresses

\end{document}